\documentclass[a4paper,11pt]{amsart}

\usepackage{comment}

\usepackage{etoolbox}
\ifdef{\crop}{%
\usepackage[includeheadfoot,twoside=False,paperwidth=448pt,paperheight=587pt,rmargin=15pt,lmargin=15pt,tmargin=15pt,bmargin=15pt]{geometry}%
}{%
\setlength{\topmargin}{22mm}
\addtolength{\topmargin}{-1in}
\setlength{\oddsidemargin}{32mm}
\addtolength{\oddsidemargin}{-1in}
\setlength{\evensidemargin}{32mm}
\addtolength{\evensidemargin}{-1in}
\setlength{\textwidth}{146mm}
\setlength{\textheight}{240mm}
}%

\usepackage[active]{srcltx}
\usepackage{amsmath,amsthm,amsxtra}
\usepackage{amssymb}

\usepackage[mathscr]{eucal}

\usepackage{bm}
\usepackage[all]{xy}

\usepackage{aliascnt}

\theoremstyle{plain}
\newtheorem{thm}{Theorem}[section]
\newtheorem*{thm*}{Theorem}

\newaliascnt{prop}{thm}
\newaliascnt{cor}{thm}
\newaliascnt{lem}{thm}
\newaliascnt{claim}{thm}
\newaliascnt{defn}{thm}
\newaliascnt{ques}{thm}
\newaliascnt{conj}{thm}
\newaliascnt{fact}{thm}
\newaliascnt{rem}{thm}
\newaliascnt{ex}{thm}
\newtheorem{prop}[prop]{Proposition}
\newtheorem{cor}[cor]{Corollary}
\newtheorem{lem}[lem]{Lemma}
\newtheorem{claim}[claim]{Claim}
\newtheorem*{prop*}{Proposition}
\newtheorem*{cor*}{Corollary}
\newtheorem*{lem*}{Lemma}
\newtheorem*{claim*}{Claim}
\theoremstyle{definition}

\newtheorem{ques}[ques]{Question}
\newtheorem{conj}[conj]{Conjecture}
\newtheorem*{defn*}{Definition}
\newtheorem*{ques*}{Question}
\newtheorem*{conj*}{Conjecture}

\newtheorem*{prob*}{Problem}

\newtheorem{rem}[rem]{Remark}

\newtheorem*{fact*}{Fact}
\newtheorem*{rem*}{Remark}
\newtheorem*{ex*}{Example}
\aliascntresetthe{prop}
\aliascntresetthe{cor}
\aliascntresetthe{lem}
\aliascntresetthe{claim}
\aliascntresetthe{defn}
\aliascntresetthe{ques}
\aliascntresetthe{conj}
\aliascntresetthe{fact}
\aliascntresetthe{rem}
\aliascntresetthe{ex}

\usepackage[pointedenum]{paralist}
\usepackage{varioref}
\labelformat{equation}{\textnormal{(#1)}}
\labelformat{enumi}{\textnormal{(#1)}}

\usepackage[a4paper]{hyperref}

\def\textsectionN~{\textsection{}}

\renewcommand\phi{\varphi}
\renewcommand\epsilon{\varepsilon}
\renewcommand\leq{\leqslant}
\renewcommand\geq{\geqslant}
\renewcommand\subsetneq{\subsetneqq}

\makeatletter
\newcommand{\set}{  \@ifstar{\@setstar}{\@set}}\newcommand{\@setstar}[2]{\{\, #1 \mid #2 \,\}}
\newcommand{\@set}[1]{\{ #1 \}}
\makeatother

\newcommand{\trans}[1][1]{\raisebox{#1ex}{\scriptsize\kern0.1em$t$\kern-0.1em}}

\DeclareMathOperator{\mult}{mult}

\DeclareMathOperator{\Supp}{Supp}

\DeclareMathOperator{\lct}{lct}
\DeclareMathOperator{\coeff}{coeff}

\def\Q{\mathbb{Q}}
\def\R{\mathbb{R}}
\def\C{\mathbb{C}}

\def\r+{\mathbb{R}_{\geq 0}}

\def\ep{\varepsilon}

\def\r+{{\R}_{\geq 0}}
\def\q+{{\Q}_{\geq 0}}

\def\arw{\rightarrow}
\def\*c{\C^{\times}}

\def\C{\mathbb {C}}

\def\Q{\mathbb {Q}}
\def\R{\mathbb {R}}

\newcommand{\cali}{\mathcal {I}}
\newcommand{\calj}{\mathcal {J}}

\newcommand{\calo}{\mathcal {O}}

\makeatletter
  
  \@addtoreset{equation}{section}
\makeatother

\title[A remark on higher syzygies on abelian surfaces]{A remark on higher syzygies on abelian surfaces}

\author[A.~Ito]{Atsushi~Ito}
\address{Graduate School of Mathematics,
Nagoya University,
Nagoya, Japan}
\email{atsushi.ito@math.nagoya-u.ac.jp}


\begin{document}

\maketitle

\begin{abstract}
In this note, we give a slight improvement of a result of A.~K\"{uronya} and V.~Lozovanu
about higher syzygies on abelian surfaces.
\end{abstract}

\section{Introduction}

We work over the complex number field $\C$.
In \cite{KL},
A.~K\"{u}ronya and V.~Lozovanu give the following Reider-type theorem for higher syzygies on abelian surfaces:

\begin{thm}[{\cite[Theorem 1.1]{KL}}]\label{thm_KL}
Let $p$ be a non-negative integer,
$X$ be an abelian surface,
and $L$ be an ample line bundle on $X$ with $(L^2) \geq 5(p+2)^2$.
Then the following are equivalent.
\begin{enumerate}
\item $X$ does not contain an elliptic curve $C$ with $(L.C) \leq p+2$, 
\item Property $(N_p)$ holds for $L$.
\end{enumerate}
\end{thm}

We refer the readers to \cite{Ei}, \cite[Chapter 1.8.D]{La} for the definition of property ($N_p$).
We just note here that ($N_p$)'s consist an increasing sequence of positivity properties.
For example, 
($N_0$) holds for $L$ if and only if $L$ defines a projectively normal embedding,
and ($N_1$) holds if and only if ($N_0$) holds and the homogeneous ideal of the embedding is generated by quadrics. 
In the note,
we show a slight improvement of \autoref{thm_KL} as follows.

\begin{thm}\label{main_rem}
In \autoref{thm_KL},
it is enough to assume $(L^2) > 4(p+2)^2$ instead of $(L^2) \geq 5(p+2)^2$.
\end{thm}

\begin{rem}
In \cite{AKL},
the equivalence of (1) and (2) in
\autoref{thm_KL} is proved for $K3$ surfaces under the assumption $(L^2) > \frac12 (p+4)^2$.
\end{rem}

In \cite{KL}, $(2) \Rightarrow (1)$ in \autoref{thm_KL} is proved under the assumption
$(L^2) \geq 4p+5$.
Hence to prove \autoref{main_rem},
it suffices to show the converse $(1) \Rightarrow (2)$ under the assumption $(L^2) > 4(p+2)^2$.

To show $(1) \Rightarrow (2)$,
we use the following result in \cite{LPP} (see also \cite[Section 3]{KL}):

\begin{thm}[\cite{LPP}]\label{thm_LPP}
Let $p$ be a non-negative integer,
$X$ be an abelian variety,
and $L$ be an ample line bundle on $X$ 
such that there exists an effective 
$\Q$-divisor $F_0$ on $X$ satisfying
\begin{itemize}
\item[(i)] $\frac{1}{p+2}L-F_0$ is ample,
\item[(ii)] the multiplier ideal $\calj(X,F_0)$ coincides with the maximal ideal 
$\cali_o$ of the origin $o \in X$.
\end{itemize}
Then $(N_p)$ holds for $L$.
\end{thm}

To construct such a divisor $F_0$,
K\"{uronya} and Lozovanu use
Okounkov bodies as a key tool.
Instead, we use a standard technique developed in the study of Fujita's base point freeness conjecture \cite{EL}, \cite{He}, \cite{Ka} \cite{Ko}, etc.

We also note that G.~Pareschi \cite{Pa} proved a conjecture of R.~Lazarsfeld
which claims that for an ample line bundle $L$ on an abelian variety,
 $(N_p)$ holds for $L^{\otimes p+3}$.
On the other hand, results in \cite{KL}, \cite{LPP} can prove ($N_p$) for $L$ which is not necessarily a multiple of a line bundle.

\subsection*{Acknowledgments}
The author was supported by Grant-in-Aid 
for Scientific Research (No.\ 26--1881, 17K14162).

\section{Preliminaries}

We recall some notations. We refer the reader to \cite{KM}, \cite[Chapter 9]{La2} for detail.

Let $X$ be a smooth variety 
and $D=\sum_i d_i D_i$ be an effective $\Q$-divisor on $X$. 
Let $f: Y \arw X $ be the log resolution of $(X,D)$
and write 
\[
K_Y = f^* (K_X +D) + F 
\]
with $F=\sum_j b_j F_j$.
Here we assume that if $F_j$ is the strict transform of $D_i$,
we take $b_j=-d_i$,
and all the other $F_j$'s are exceptional divisors of $f$.

The pair $(X,D)$ is called \emph{log canonical}  (resp.\ 
\emph{klt}) at $x \in X$ 
if $b_j \geq -1$ (resp.\ $b_j > -1$) for any log resolution $f$ and any $j$ with $x \in f(F_j)$.
If $(X,D)$ is log canonical at $x$ and $b_j > -1$ for any log resolution $f$ and any \emph{exceptional} $F_j$ with $x \in f(F_j)$,
$(X,D)$ is called 
\emph{plt} at $x$.
The pair $(X,D)$ is called log canonical (resp.\ klt, plt) if $(X,D)$ is log canonical (resp.\ klt, plt) at any $x \in X$.

The \emph{multiplier ideal} $\calj(X,D)$ of $(X,D)$ is defined as
\[
\calj(X,D) := f_* \calo_Y(\lceil F \rceil),
\]
which does not depend on the choice of the log resolution $f$.

The \emph{log canonical threshold} of $D$ is defined to be
\[
\lct(D)= \max\{ s \geq 0 \, | \, (X,s D) \text{ is log canonical}\}.
\]
Similarly,
the log canonical threshold of $(X,D)$ at $x \in X$ is 
\[
\lct_x(D)= \max\{ s \geq 0 \, | \, (X,s D) \text{ is log canonical at } x\}.
\]

A subvariety $Z$ of $X$ is called a \emph{log canonical center} of $(X,D)$
if there exists a log resolution $f$ of $(X,D) $ and some $j$ such that $b_j \leq -1$ and $f(F_j)=Z$.
We note that a log canonical pair $(X,D)$ is plt if and only if $(X,D)$ has no log canonical center of codimension $\geq 2$.


\begin{rem}\label{rem_max_ideal}
If $\{x\}$ is the \emph{unique} log canonical center of $(X,D)$ containing $x$,
we have $\Supp \calo_X / \calj(X,D) =\{x\}$ around $x$.
Furthermore,
if $(X,D) $ is log canonical at $x \in X$,
every negative coefficient of $\lceil F \rceil$ over $x$ is $-1$.
Hence $\calj(X,D)$ is reduced at $x$,
that is,
$\calj(X,D) = \cali_x$ holds around $x$ in such a case.
\end{rem}

The following theorem about the existence of minimal log canonical centers is known
(see \cite{EL},\cite{He},\cite{Ka}).
We note that for surface this theorem can be easily proved.

\begin{thm}\label{thm_minimal _center}
Let $X$ be a smooth variety 
and $D$ be an effective $\Q$-divisor on $X$ such that $(X,D) $ is log canonical.
Then every irreducible component of the intersection of two log canonical centers of $(X,D)$
is also a log canonical center of $(X,D)$.

In particular, if $(X,D) $ is log canonical and not klt at $x \in X$,
there exists the unique minimal log canonical center $Z$ of $(X,D)$ containing $x$.
Furthermore, $Z$ is normal at $x$. 
\end{thm}

\section{Proof}

\subsection{Proof of \autoref{main_rem}}

As stated in Introduction,
we will construct a divisor $F_0$ in \autoref{thm_LPP}.
In fact, we can construct such a divisor easily by a standard technique developed in the study of Fujita's base point freeness conjecture. 
Roughly speaking,
we cut down minimal log canonical centers to obtain a $0$-dimensional log canonical center.
In higher dimensional case,
the arguments and suitable explicit estimation are difficult and complicated,
but they are relatively easy for surfaces (see \cite[Section 2]{EL}).

Throughout this subsection,
$X$ is an abelian surface
and $\pi : X' \arw X $ is the blow-up at the origin $o \in X$.
We denote by $E \subset X'$ the exceptional divisor.

\vspace{2mm}

\begin{prop}\label{main_lem}
Let 
$B$ be an ample $\Q$-divisor on $X$.
If $\pi^* B -2E $ is big and $(B.C) > 1$ holds for any elliptic curve $C \subset X$,
there exists an effective $\Q$-divisor $F_0$ on $X$ such that $B- F_0$ is ample and $\calj(X,F_0)=\cali_o$.
\end{prop}

\begin{rem}
If $X$ is a general surface,
we usually assume $(B.C) \geq 2$ for a curve $C$
to construct a divisor $F_0$ as in \autoref{main_lem} (see \cite[1.10 Variant]{EL}, for example).
It is not so surprising that we can relax the numerical condition for abelian surfaces.

\end{rem}

By \autoref{rem_max_ideal},
we want $F_0$ such that $\{o\}$ is the unique log canonical center of a log canonical pair $(X,F_0)$.
Considering a small perturbation,
it suffices to find an effective $\Q$-divisor with a $0$-dimensional minimal center:

\begin{lem}\label{lem_minimal_center}
To prove \autoref{main_lem},
it suffices to construct an effective $\Q$-divisor $F$ on $X$ such that $B-F$ is ample,
$(X,F)$ is log canonical and has a $0$-dimensional minimal center.
\end{lem}

\begin{proof}
Let $\{x\} \subset X$ be a $0$-dimensional minimal center of $(X,F)$.
Choose a general effective divisor $F'$ on $X$ which contains $x$.
By replacing $F$ with $c_{\ep} F + \ep F'$ for $0 < \ep \ll 1$,
where $c_{\ep}= \max\{ s \geq 0 \, | \, (X,s F+\ep F') \text{ is log canonical}\}$,
we may assume that $\{x\} \subset X$ is the unique minimal center of $(X,F)$
as in the proof of \cite[Proposition 2.4]{He}.
We note that the ampleness of $B - F$ is preserved if $\ep$ is sufficiently small.
Then we have $\calj(X,F)=\cali_x$ by \autoref{rem_max_ideal}. 
Set $F_0=t_x^* F$ to be the pull back of $F$ by the automorphism
\[
t_x :X \rightarrow X \quad : \quad  y \mapsto y+x .
\]
Then $B - F_0 \equiv B - F $ is ample and $\calj(X,F_0)=\cali_o$ holds.
\end{proof}

\begin{proof}[Proof of \autoref{main_lem}]
We will construct $F$ as in \autoref{lem_minimal_center}.

Since bigness is an open condition,
$\pi^* B - t E $ is big for a rational number $t $ if $0 < t-2 \ll 1$.
Fix such $t$
and let $\pi^* B - t E =P_t+N_t$ be the Zariski decomposition.
Since $P_t$ is nef and big,
we can find an effective divisor $N'$ such that $P_t-\frac{1}{k} N'$ is ample for $k \gg 0$ by \cite[Theorem 2.3.9]{La}.
For $k \gg 0$,
we choose a general effective $\Q$-divisor $A \equiv P_t-\frac{1}{k} N'$ and set
\begin{align}\label{eq_def_D}
D' := A + \frac{1}{k} N' +N_t \equiv \pi^* B - t E , \quad D:=\pi_* D' \equiv B.
\end{align}
Then $D$ is an effective $\Q$-divisor on $X$ with the multiplicity $\mult_o (D) \geq t$.
Hence we have
\[
c:= \lct(X,D) \leq \lct_o (X,D) 
\leq \frac{2}{t} <1.
\]

If there exists a $0$-dimensional minimal center $\{x\}$ of $(X,cD)$,
we can take $F:= cD$, which satisfies the condition in \autoref{lem_minimal_center}.
Hence we may assume that there exists no $0$-dimensional minimal center of $(X,cD)$.
Then there exists a $1$-dimensional minimal center $C \subset X$ of $(X,cD)$.

\begin{claim}\label{clm1}
The minimal center $C$ is an elliptic curve containing the origin $o$.
\end{claim}

\begin{proof}
Since $ C$ is a log canonical center of $(X,cD)$,
the coefficient 
of $C$ in $D$ is $\frac{1}{c} >1$. 
Since $A$ is general and $k$ is sufficiently large in \ref{eq_def_D},
the strict transform $C' \subset X'$ of $C$ is contained in $\Supp N_t$.
Hence $C' $ is a negative curve
and we have 
\begin{align}\label{eq_ineq}
0 > (C'^2)= ((\pi^*C - m E)^2)=(C^2) -m^2,
\end{align}
where $m := \mult_o(C)  \geq 0$. 

Since $C $ is a curve on an abelian surface,
$(C^2) \geq 0$ holds.
On the other hand,
$m=0$ or $1$
since a minimal center is normal by \autoref{thm_minimal _center}
and hence $C$ is smooth.
Thus $(C^2)=0 $ and $m=1$ hold by \ref{eq_ineq}.
Since $X$ is an abelian surface,
$(C^2)=0$ implies that $C$ is an elliptic curve.
\end{proof}

We cut down the minimal log canonical center $C$.
We refer the reader to \cite{He}, \cite{Ka}, \cite[Section 10.4]{La2} for detail.

By assumption and \autoref{clm1},
we have
\[
(B-C.C) =(B.C) -(C^2) = (B.C) >1.
\]
Thus there exists an effective $\Q$-divisor $\overline{D}_1 \equiv (B-C)|_C $ on $C$
with $\mult_o  (\overline{D}_1) >1$.
We take such $\overline{D}_1 $ 
so that $\mult_p  (\overline{D}_1) < 1 $ for any $p \in C \setminus \{o\}$.

Since the coefficient of $C$ in $cD$ is $1$,
$cD - C$ is effective hence nef because $X$ is an abelian surface.
Thus 
\[
B-C =(1-c) B +(cB-C) 
\equiv (1-c)B+(cD-C)
\]
is ample by $c <1$.
Hence
we can take an effective $\Q$-divisor $D_1 \equiv B-C$ on $X$
such that $C \not \subset \Supp D_1$ and $D_1|_C = \overline{D}_1$ as in Step 2 in the proof of \cite[Proposition 3.2]{He}.
We take general such $D_1$,  
then $(X,C+D_1)$ is klt outside $C$.
Since  $\mult_o  (D_1|_C)  = \mult_o  (\overline{D}_1) >1$,
$(X,C+D_1)$ is not log canonical at $o \in X$ by adjunction \cite[Theorem 5.50]{KM}.
On the other hand, since $\mult_p  (D_1|_C)   =\mult_p  (\overline{D}_1) < 1 $ for any $p \in C \setminus \{o\}$,
$(X,C+D_1) $ is plt in a neighborhood of $C \setminus \{o\}$ by inversion of adjunction \cite[Theorem 5.50]{KM}.
Set
\[
F=C+c_1 D_1 ,
\]
where $c_1 :=\max \{ s \geq 0 \, | \, (X,C+s D_1) \text{ is log canonical}\} <1$.
Then $\{o\} \subset X$ is the minimal log canonical center of $(X,F)$.
Since 
\[
B-F= B- (C+c_1 D_1) \equiv (1-c_1)(B-C)
\]
is ample by $c_1 <1$,
this $F$ satisfies the condition in \autoref{lem_minimal_center}.
Hence
\autoref{main_lem} is proved.
\end{proof}


\begin{cor}\label{thm_suff}
Let $p$ be a non-negative integer
and $L$ be an ample line bundle on $X$.
If $\pi^* L -2(p+2)E$ is big and $(L.C) > p+2$ holds for any elliptic curve $C \subset X$,
then $(N_p)$ holds for $L$.
\end{cor}

\begin{proof}
We can apply \autoref{main_lem} to $B=\frac{1}{p+2} L$.
Then Property ($N_p$) holds for $L$ by \autoref{thm_LPP} and \autoref{main_lem}.
\end{proof}


\begin{proof}[Proof of \autoref{main_rem}]
As stated in Introduction,
it suffices to show $(1) \Rightarrow (2)$ under $(L^2) > 4(p+2)^2$.
In this case,
$\pi^* L -2 (p+2)E$ is big.
Hence $(1) \Rightarrow (2)$ follows from \autoref{thm_suff}
\end{proof}

\subsection{Remark}

By a similar argument as in the proof of \autoref{main_lem},
we can show the following lemma,
which recovers  \cite[Theorem 2.5 (1)]{KL}.
We denote by $\coeff_C(D)$ the coefficient of a prime divisor $C$ in a $\Q$-divisor $D$.

\begin{lem}
Let $X$ be a smooth projective surface and $B$ be an ample $\Q$-divisor on $X$.
Let $\pi : X' \arw X $ be the blow-up at a point $x \in X$ and $E \subset X'$ be the exceptional divisor.
Assume that $\pi^* B -2E $ is big
and let $ \pi^* B -2E  =P+N$ be the Zariski decomposition.
If 
\begin{align}\label{eq_condition}
\coeff_{C'} (N) < 1 \quad \text{ for any prime divisor } C'  \text{ on } X' \text{ with } (C'.E)=1,
\end{align}
there exists an effective $\Q$-divisor 
$F_0 \equiv c_0 B$ for some $0 < c_0 < 1$ such that $\calj(X,F_0 ) = \cali_x$ holds in a neighborhood of the point $x$.
\end{lem}

\begin{proof}
As in the proof of \autoref{main_lem},
we take $t$ with $0 < t-2 \ll 1$ and let $\pi^* B -t E =P_t +N_t$ be the Zariski decomposition.
We note that $N_t$ also satisfies \ref{eq_condition} by \cite[Proposition 1.16]{BKS},
i.e.\
$\coeff_{C'} (N_t) < 1$ holds if $(C'.E) = 1$.
We set $D'$ on $X'$ and $D=\pi_* D'$ on $X$ as \ref{eq_def_D}.
We can take $D'$ so that $\coeff_{C'} (D') < 1$ holds if $(C'.E) = 1$.
Since $\mult_x (D) \geq t >2$, we have
$
c:= \lct_x(X,D) 
\leq \frac{2}{t} <1.
$

Assume that the minimal center of $(X,cD)$ at $x$ is a curve $C$.
Then $C $ is smooth at $x$ and 
hence we have $(C'.E) =1$ for the strict transform $C'$ of $C$.
Thus $\coeff_C (D)= \coeff_{C'} (D') < 1$ holds by the definition of $D'$.
Since $c <1$, we have $\coeff_C (c D)  < 1$,
which contradicts the assumption that $C$ is the minimal center of $(X,cD)$.



Thus the minimal center of $(X,cD)$ at $x$ is $\{x\}$.
As in the proof of \autoref{lem_minimal_center},
we can take $F_0$ as a small perturbation of $cD \equiv cB$.
\end{proof}

\section{Naive Questions}

We end this note by questions in higher dimensions.
As a generalization of Fujita's base point freeness conjecture \cite{Fu},
the following conjecture is known:

\begin{conj}[{\cite[5.4 Conjecture]{Ko2}}]\label{conj_fujita}
Let $Y$ be a smooth projective variety, $y \in Y$ be a point and
$L$ be a nef and big line bundle on $Y$ . Assume that if $y \in Z \subsetneq Y$ is an irreducible
subvariety then
\[
(L^{\dim Z}.Z) \geq (\dim Y)^{\dim Z}, \quad \text{and} \quad (L^{\dim Y}) > (\dim Y)^{\dim Y}. 
\]
Then $K_Y + L$ is base point free at $y$.
\end{conj}


The condition (1) in \autoref{thm_KL} and the condition $(L^2) > 4(p+2)^2$ in \autoref{main_rem} are equivalent to
\[
(B.C) >1=(\dim C)^{\dim C}, \quad \text{and} \quad  (B^2) > 4=(\dim X)^{\dim X} 
\]
respectively,
where $B=\frac{1}{p+2} L$ and $C \subset X$ is any elliptic curve.
Naively,
we have the following question as an analog of \autoref{conj_fujita}:

\begin{ques}\label{ques1}
Let $p$ be a non-negative integer,
$X$ be an abelian variety,
and $L$ be an ample line bundle on $X$.
Set $B=\frac{1}{p+2} L$.
Assume that $(B^{\dim Z} .Z) > (\dim Z)^{{\dim Z}}$ holds for any abelian subvariety $\{o\} \neq  Z \subset X$.
Then does property ($N_p$) hold for $L$?
\end{ques}

We note that in this question we assume $(B^{\dim Z} .Z)   > (\dim Z)^{{\dim Z}} $, which is weaker than $(B^{\dim Z} .Z)  \geq (\dim X)^{\dim Z}$ for $Z \subsetneq X$.
By \autoref{thm_suff},
we also have a little stronger version:

\begin{ques}\label{ques2}
Let $p$ be a non-negative integer,
$X$ be an abelian variety,
and $L$ be an ample line bundle on $X$.
Let $\pi : X' \arw X $ be the blow-up at the origin $o \in X$.
Set $B=\frac{1}{p+2} L$.
Assume that for any abelian subvariety $\{o\} \neq  Z \subset X$, $\left( \pi^*B - (\dim Z)E \right) |_{Z'} $ is big, where $Z' \subset X'$ is the strict transform of $Z$.
Then does property ($N_p$) hold for $L$?
\end{ques}

\begin{rem}

If $X$ is simple, that is, there exists no abelian subvariety $\{o\} \neq Z \subsetneq X$,
then $(B^{g}) > g^g$ is  the unique condition in \autoref{ques1} for $g=\dim X$.
The following are known in any dimension.

\begin{enumerate}
\item By \cite[Corollary B]{LPP},
 ($N_p$) holds for $L$
if $(X,L)$  is very general 
and $(B^g) > (4g)^g /2$.
\item J.~Iyer studied the projective normality, i.e.\ ($N_0$), for simple abelian varieties. 
By \cite[Theorem 1.2]{Iy},
 if $X$ is simple and $(B^g) >( g{!} )^2$ for $B=\frac12 L$, 
then $L$ satisfies ($N_0$).
We note that 
\[
( g{!} )^2 = \prod_{i=1}^g i (g+1-i) \geq  \prod_{i=1}^g g =g^g.
\]
\end{enumerate}
\end{rem}

\end{document}